\documentclass[reqno, 12pt]{amsart}
\pdfoutput=1
\makeatletter
\let\origsection=\section \def\section{\@ifstar{\origsection*}{\mysection}}
\def\mysection{\@startsection{section}{1}\z@{.7\linespacing\@plus\linespacing}{.5\linespacing}{\normalfont\scshape\centering\S}}
\makeatother

\usepackage{amsmath,amssymb,amsthm}
\usepackage{mathrsfs}
\usepackage{mathabx}\changenotsign
\usepackage{dsfont}

\usepackage{xcolor}
\usepackage[backref=section]{hyperref}
\usepackage[ocgcolorlinks]{ocgx2}
\hypersetup{
	colorlinks=true,
	linkcolor={red!60!black},
	citecolor={green!60!black},
	urlcolor={blue!60!black},
}

\usepackage[open,openlevel=2,atend]{bookmark}

\usepackage[abbrev,msc-links,backrefs]{amsrefs}
\usepackage{doi}

\renewcommand{\PrintDOI}[1]{\doi{#1}}

\usepackage[T1]{fontenc}
\usepackage{lmodern}
\usepackage[babel]{microtype}
\usepackage[english]{babel}

\usepackage{geometry}
\numberwithin{equation}{section}
\numberwithin{figure}{section}

\usepackage{enumitem}

\let\polishlcross=\l
\def\l{\ifmmode\ell\else\polishlcross\fi}

\let\setminus=\smallsetminus

\makeatletter
\def\moverlay{\mathpalette\mov@rlay}
\def\mov@rlay#1#2{\leavevmode\vtop{   \baselineskip\z@skip \lineskiplimit-\maxdimen
		\ialign{\hfil$\m@th#1##$\hfil\cr#2\crcr}}}
\newcommand{\charfusion}[3][\mathord]{
	#1{\ifx#1\mathop\vphantom{#2}\fi
		\mathpalette\mov@rlay{#2\cr#3}
	}
	\ifx#1\mathop\expandafter\displaylimits\fi}
\makeatother

\newcommand{\dcup}{\charfusion[\mathbin]{\cup}{\cdot}}

\DeclareFontFamily{U}  {MnSymbolC}{}
\DeclareSymbolFont{MnSyC}         {U}  {MnSymbolC}{m}{n}
\DeclareFontShape{U}{MnSymbolC}{m}{n}{
	<-6>  MnSymbolC5
	<6-7>  MnSymbolC6
	<7-8>  MnSymbolC7
	<8-9>  MnSymbolC8
	<9-10> MnSymbolC9
	<10-12> MnSymbolC10
	<12->   MnSymbolC12}{}
\DeclareMathSymbol{\powerset}{\mathord}{MnSyC}{180}
\DeclareMathSymbol{\YY}{\mathord}{MnSyC}{42}

\usepackage{tikz}
\usetikzlibrary{calc,decorations.pathmorphing}
\usetikzlibrary{arrows,decorations.pathreplacing}
\pgfdeclarelayer{background}
\pgfdeclarelayer{foreground}
\pgfdeclarelayer{front}
\pgfsetlayers{background,main,foreground,front}

\usepackage{multicol}
\usepackage{subcaption}
\let\epsilon=\varepsilon
\let\eps=\epsilon
\let\phi=\varphi
\let\rho=\varrho
\let\theta=\vartheta

\def\ZZ{{\mathds Z}}

\newtheoremstyle{note}  {4pt}  {4pt}  {\sl}  {}  {\bfseries}  {.}  {.5em}          {}
\newtheoremstyle{introthms}  {3pt}  {3pt}  {\itshape}  {}  {\bfseries}  {.}  {.5em}          {\thmnote{#3}}
\newtheoremstyle{remark}  {2pt}  {2pt}  {\rm}  {}  {\bfseries}  {.}  {.3em}          {}

\theoremstyle{plain}
\newtheorem{thm}{Theorem}[section]
\newtheorem{lem}[thm]{Lemma}

\theoremstyle{note}

\theoremstyle{remark}

\begin{document}

\title[Note on the Theorem of Balog, Szemer\'edi, and Gowers]
{Note on the Theorem of Balog, Szemer\'edi, and Gowers}

\author{Christian Reiher}
\address{Fachbereich Mathematik, Universit\"at Hamburg, Hamburg, Germany}
\email{Christian.Reiher@uni-hamburg.de}

\author{Tomasz Schoen}
\address{A. Mickiewicz University, Department of Discrete Mathematics, Pozna\'n, Poland}
\email{schoen@amu.edu.pl}

\begin{abstract}
	We prove that every additive set $A$ with energy $E(A)\ge |A|^3/K$ 
	has a subset $A'\subseteq A$ of size $|A'|\ge (1-\eps)K^{-1/2}|A|$  
	such that $|A'-A'|\le O_\eps(K^{4}|A'|)$. This is, essentially,
	the largest structured set one can get in the Balog-Szemer\'edi-Gowers theorem. 
\end{abstract}

\maketitle

\setcounter{footnote}{1}

\section{Introduction}

An {\it additive set} is a nonempty finite subset of an abelian group.
The {\it energy} of an additive set $A$ is defined to be 
the number $E(A)$ of quadruples $(a_1, a_2, a_3, a_4)\in A^4$ solving the 
equation $a_1+a_2=a_3+a_4$. 
An easy counting argument shows
\begin{equation}\label{e:1}
	E(A)=\sum_{d\in A-A}r_{A-A}(d)^2\,,
\end{equation}
where $r_{A-A}(d)$ indicates the number of representations of $d$ as a difference 
of two members of $A$. So the Cauchy-Schwarz inequality yields $E(A)\ge |A|^4/|A-A|$ 
and, in particular, every additive set $A$ with small difference set $A-A$ contains 
a lot of energy. 
In the converse direction Balog and Szemer\'edi~\cite{BS} proved that large energy 
implies the existence of a substantial subset whose difference set is small. After 
several quantitative  improvements (see e.g., Gowers~\cite{G} and Balog~\cite{B}) 
the hitherto best version of this result was obtained by the second author~\cite{S}. 

\begin{thm}\label{s:0}
	Given a real $K\ge 1$ every additive set $A$ with energy $E(A)\ge |A|^3/K$ 
	has a subset $A'\subseteq A$ of size $|A'|\ge \Omega(|A|/K)$ 
	such that $|A'-A'|\le O(K^{4}|A'|)$. \qed
\end{thm}

When investigating the question how a quantitatively optimal version of this 
result might read, there are two different directions one may wish to pursue. 
First, there is the obvious problem whether the exponent $4$ can be replaced 
by some smaller number. Second, one may try to find ``the largest'' set 
$A'\subseteq A$ such that $|A'-A'|\le O_K(|A'|)$ holds. As the following 
example demonstrates, there is no absolute constant $\eps_\star>0$  
such that $|A'|\ge (1+\eps_\star)K^{-1/2}|A|$ can be achieved in general.  

Fix an arbitrary natural number $n$. For a very large finite abelian group~$G$ 
we consider the additive set
\[
	A=\bigl\{(g_1, \ldots, g_n)\in G^n\colon \text{ there is at most one index $i$ 
	such that $g_i\ne 0$}\bigr\}
\]
whose ambient group is $G^n$. Obviously we have  
\[
	|A|=|G|n+O_n(1)
	\quad \text{ and } \quad
	E(A)=|A|^3/n^2+O_n(|A|^2)\,,
\]
so the real number $K$ satisfying $E(A)=|A|^3/K$ is roughly $n^2$. 
However, every $A'\subseteq A$ of size $|A'|\ge (1+\eps)|G|$ satisfies 
$|A'-A'|\ge \eps^2 |G|^2$. 
Our main result implies that this is, in some sense, already the worst example. 
More precisely, for every fixed $\eps>0$ the Balog-Szemer\'edi-Gowers theorem 
holds with $|A'|\ge (1-\eps)K^{-1/2}|A|$. Perhaps surprisingly, we can also 
reproduce the best known factor $K^4$. 
 	
\begin{thm}\label{s:1} 
	Given real numbers $K\ge 1$, $\eps\in (0, 1/2)$, and an additive set $A$ 
	with energy $E(A)\ge |A|^3/K$ there is a subset $A'\subseteq A$ such that 
		\[
		|A'|\ge (1-\eps)K^{-1/2}|A|
		\quad \text{ and } \quad 
		|A'-A'|\le 2^{33}\eps^{-9}K^{4}|A'|=O_\eps(K^4|A'|)\,.
	\]
	\end{thm}

Our proof has two main cases and in one of them (see Lemma~\ref{l:P} below)
we even get the stronger bound $|A'-A'|\le O_\eps(K^3|A'|)$. It would be interesting
to prove this in the second case as well. 
Using examples of the form $A=\{x\in \ZZ^d\colon \|x\|\le R\}$ one can show that 
the exponent~$4$ cannot be replaced by any number smaller than 
$\log(4)/\log(27/16)\approx 2.649$ (see~\cite{Shao}).

\section{Preliminaries}

This section discusses two auxiliary results we shall require for the proof 
of Theorem~\ref{s:1}.
The first of them is similar to~\cite{TV}*{Lemma~6.19}.  

\begin{lem}\label{lem:tv}
	If $\delta, \xi\in (0, 1]$ and $R\subseteq A^2$ denotes a binary relation 
	on a set~$A$ such that $|R|\ge \delta |A|^2$, then there is a set $A'\subseteq A$ 
	of size $|A'|\ge \delta(1-\xi) |A|$ which possesses the following property: 
	For every pair $(a_1, a_2)\in A'^2$ there are at least $2^{-7}\delta^4\xi^4|A|^2|A'|$
	triples $(x, b, y)\in A^3$ such that $(a_1, x), (b, x), (b, y), (a_2, y)\in R$. 
\end{lem}

\begin{proof}
	Set $N(x)=\{a\in A\colon (a, x)\in R\}$ for every $x\in A$. 
	Since $\sum_{x\in A}|N(x)|=|R|\ge \delta |A|^2$, the Cauchy-Schwarz inequality yields
		\begin{equation}\label{eq:4254}
		\sum_{x\in A}|N(x)|^2\ge \delta^2|A|^3\,.
	\end{equation}
		Setting $K(a, a')=\{x\in A\colon a, a'\in N(x)\}$ for every pair $(a, a')\in A^2$
	and 
		\begin{align*}
		\Omega
		=
		\bigl\{(a, a')\in A^2\colon |K(a, a')|\le \delta^2\xi^2|A|/8\bigr\}
	\end{align*}
		a double counting argument yields  
		\[
		\sum_{x\in A}|N(x)^2\cap \Omega|
		=
		\sum_{(a, a')\in \Omega}|K(a, a')|
		\le
		\delta^2\xi^2|A||\Omega|/8
		\le
		\delta^2\xi^2|A|^3/8\,.
	\]
		Together with~\eqref{eq:4254} we obtain 
		\[
		\sum_{x\in A} \bigl(|N(x)|^2-8\xi^{-1}|N(x)^2\cap \Omega|\bigr)
		\ge
		\delta^2(1-\xi)|A|^3
	\]
		and, hence, there exists some $x_\star\in A$ such that the set $A_\star=N(x_\star)$
	satisfies 
		\begin{equation}\label{eq:4819}
		|A_\star|^2-8\xi^{-1}|A_\star^2\cap\Omega|\ge \delta^2(1-\xi)|A|^2\,.
	\end{equation}
		We shall prove that the set 
		\[
		A'=\{a\in A_\star\colon \text{ the number of all $a'\in A_\star$ 
			with $(a, a')\in \Omega$ is at most $|A_\star|/4$}\}
	\]
		has all required properties. 
	By~\eqref{eq:4819} we have 
		\[
		|A_\star\setminus A'||A_\star|/4
		\le 
		|A_\star^2\cap\Omega|
		\le 
		\xi|A_\star|^2/8\,,
	\]
	for which reason 
		\[
		|A'|
		\ge 
		(1-\xi/2)|A_\star|
		\ge 
		(1-\xi)^{1/2}|A_\star|
		\overset{\eqref{eq:4819}}{\ge}
		\delta(1-\xi) |A|\,,
	\]
		meaning that $A'$ is indeed sufficiently large. To conclude the proof we need to show 
		\[
		\sum_{b\in A}|K(a_1, b)\times K(b, a_2)| 
		\ge
		2^{-7}\delta^4\xi^4|A|^2|A'|
	\]
		for every pair $(a_1, a_2)\in A'^2$. This follows from the fact that due to 
	the definition of $A'$ there are at least $|A_\star|/2$ elements $b\in A_\star$ 
	such that the sets $K(a_1, b)$ and $K(b, a_2)$ both have at least the 
	size $\delta^2\xi^2|A|/8$.
\end{proof}

\begin{lem}\label{l:ST}
	Suppose that the real numbers $x_1, \dots, x_n\in [0, 1]$ do not vanish 
	simultaneously.
	Denote their sum by $S$ and the sum of their squares by $T$.  
	For every $\alpha\in (0, 1)$ there exists
	a set $I\subseteq [n]$ such that 		\[
		\sum_{i\in I}x_i 
		\ge 
		\max\left\{\alpha T, 
		\biggl(\frac{(1-\alpha)^5|I|^4T^4}{2^{10}S^2}\biggr)^{1/6}\right\}\,.
	\]
	\end{lem}

\begin{proof}
	For reasons of symmetry we may assume $x_1\ge\dots\ge x_n$. 
	Set $S_i=\sum_{j=1}^i x_j$ for every nonnegative $i\le n$. 
	Due to $T\le x_1 S$ and $x_1\le 1$ we have $T\le S=S_n$ and 
	thus there exists a smallest index $k\in [n]$ satisfying 
	$S_k\ge \alpha T$. Notice that
		\[
		\sum_{i=1}^{k-1}x_i^2
		\le
		\sum_{i=1}^{k-1}x_i
		=S_{k-1}
		\le \alpha T\,.
	\]
	Moreover $x_1\ge T/S$ implies the existence of a largest index $\ell$ 
	such that $x_\ell\ge (1-\alpha)T/(2S)$. 
	Due to 
		\begin{equation*}		\sum_{i=\ell+1}^n x_i^2 
		\le 
		\frac{(1-\alpha)T}{2S}\sum_{i=\ell+1}^n x_i
		\le 
		\frac{(1-\alpha)T}2\,,
	\end{equation*}
		we have 
		\begin{equation}\label{eq:ST2}
		\sum_{i=k}^{\ell} x_i^2 
		\ge  
		\frac{(1-\alpha)T}2\,,
	\end{equation}
	whence, in particular, $\ell\ge k$. Next,
		\[
		\ell \biggl(\frac{(1-\alpha)T}{2S}\biggr)^2
		\le 
		\sum_{i=1}^\ell x_i^2
		\le 
		T
	\]
		entails
		\begin{equation}\label{eq:ST3}
		(1-\alpha)^2\ell T\le 4 S^2\,. 
	\end{equation}
		
	Now assume for the sake of contradiction that our claim fails. 
	Every $i\in [k, \ell]$ satisfies $S_i\ge S_k\ge \alpha T$ 
	and thus the failure of $I=[i]$ discloses 	
		\begin{equation*}		S_i<\biggl(\frac{(1-\alpha)^5i^4T^4}{2^{10}S^2}\biggr)^{1/6}\,.
	\end{equation*}
		Combined with $ix_i\le S_i$ this entails 
		\[
		\sum_{i=k}^\ell x_i^2
		\le 
		\biggl(\frac{(1-\alpha)^5T^4}{2^{10}S^2}\biggr)^{1/3}\sum_{i=k}^\ell i^{-2/3}\,.
	\]
		In view of~\eqref{eq:ST2} we are thus led to
		\[
		\biggl(\frac{2^7S^2}{(1-\alpha)^2T}\biggr)^{1/3}
		\le 
		\sum_{i=k}^\ell i^{-2/3}
		\le
		\int_0^{\ell}x^{-2/3} \mathrm{d}x
		= 
		3\ell^{1/3}\,, 
	\]
		i.e., $2^7S^2\le 27(1-\alpha)^2\ell T$, which contradicts~\eqref{eq:ST3}.
\end{proof}

\section{The proof of Theorem~\ref{s:1}}

Let us fix two real numbers $K\ge 1$ and $\eps\in (0, 1/2)$ as well as 
an additive set $A$ satisfying $E(A)\ge |A|^3/K$. We consider the partition 
\[
	A-A=P\dcup  Q  
\]
defined by 
\begin{align*}
	P&=\bigl\{d\in A-A\colon r_{A-A}(d)\ge K^{-1/2}|A|\bigr\} \\
	\text{ and } \quad
	Q&=\bigl\{d\in A-A\colon r_{A-A}(d) < K^{-1/2}|A|\bigr\} \,.
\end{align*}
According to~\eqref{e:1} at least one of the cases 
\begin{equation}\label{eq:1113}
	 \sum_{d\in P}r_{A-A}(d)^2 \ge \frac{\eps|A|^3}{4K}
	 \quad \text{ or } \quad
	 \sum_{d\in Q}r_{A-A}(d)^2 \ge \frac{(4-\eps)|A|^3}{4K}
\end{equation}
needs to occur and we begin by analysing the left alternative.  

\begin{lem}\label{l:P}
	If $\sum_{d\in P}r_{A-A}(d)^2\ge \eps|A|^3/(4K)$, then there exists 
	a set $A'\subseteq A$ of size $|A'|\ge (1-\eps)K^{-1/2}|A|$ such that 
	$|A'-A'|\le 2^{10}\eps^{-4} K^3|A'|$.
\end{lem}

\begin{proof}
	For every difference $d\in P$ we set $A_d=A\cap (A+d)$. Due to $|A_d|=r_{A-A}(d)$
	the hypothesis implies 
		\begin{equation}\label{e:2}
		\sum_{d\in P}|A_d|^2\ge \eps|A|^3/(4K)\,.
	\end{equation}
		For every pair $(x, y)\in A^2$ the set $L(x, y)=\{d\in P\colon x, y\in A_d\}$
	has at most the cardinality $|L(x, y)|\le r_{A-A}(x-y)$, because every 
	difference $d\in L(x, y)$ corresponds to its own representation $x-y=(x-d)-(y-d)$ of
	$x-y$ as a difference of two members of $A$. Applying this observation to all pairs 
	in	
		\[
		\Xi=\bigl\{(x, y)\in A^2\colon r_{A-A}(x-y)\le \eps^2|A|/(16K)\bigr\}
	\]
		we obtain 
		\[
		\sum_{d\in P}|A_d^2\cap \Xi|
		=
		\sum_{(x, y)\in \Xi}|L(x, y)|
		\le
		\sum_{(x, y)\in \Xi}r_{A-A}(x-y)
		\le
		\frac{\eps^2|A||\Xi|}{16K}
		\le
		\frac{\eps^2|A|^3}{16K}\,.
	\]
		Together with~\eqref{e:2} this yields  
		\[
		\sum_{d\in P}\bigl(\eps|A_d^2|-4|A_d^2\cap \Xi|\bigr)\ge 0
	\]
		and, consequently, for some element $d(\star)\in P$ the set
	$A_\star=A_{d(\star)}$ satisfies $|A_\star^2\cap \Xi|\le \eps|A_\star|^2/4$. 
	We contend that the set 
		\[
		A'=\bigl\{a\in A_\star\colon \text{ There are at most $|A_\star|/4$ pairs of 
				the form $(a, x)$ in $\Xi$}\bigr\}  
	\]
		has the required properties. 
			As in the proof of Lemma~\ref{lem:tv} we obtain 
		\[
		|A'|\ge (1-\eps)|A_\star| = (1-\eps)r_{A-A}(d(\star))\ge (1-\eps)K^{-1/2}|A|\,;
	\]
		so it remains to derive the required upper bound on $|A'-A'|$. 
	
	To this end we consider an arbitrary pair $(a, a')$ of elements of $A'$. Owing to 
	the definition of $A'$ there are at least $|A_\star|/2$ elements $x\in A_\star$ 
	such that $(a, x)\not\in\Xi$ and $(a', x)\not\in\Xi$. For each of them we have
	$a-a'=(a-x)-(a'-x)$, there are at least $\eps^2|A|/(16K)$ pairs $(a_1, a_2)\in A^2$ 
	solving the equation $a-x=a_1-a_2$ and at least the same number of pairs 
	$(a_3, a_4)\in A^2$ such that $a'-x=a_3-a_4$. Altogether there are at least 
		\[
		\eps^4|A|^2|A_\star|/(2^{9}K^2)
		\ge 
		2^{-9}\eps^4K^{-5/2}|A|^3
	\]
		possibilities of writing $a-a'=(a_1-a_2)-(a_3-a_4)$ and for this reason we have 
		\[
		|A'-A'|
		\le 
		\frac{|A|^4}{2^{-9}\eps^4K^{-5/2}|A|^3}
		=
		2^9\eps^{-4}K^{5/2}|A|
		\le 
		2^{10}\eps^{-4}K^3 |A'|\,. \qedhere
	\]
	\end{proof}

We conclude the proof of Theorem~\ref{s:1} by taking care of the right case in~\eqref{eq:1113}.  

\begin{lem}\label{l:Q}
	If $\sum_{d\in Q}r_{A-A}(d)^2\ge (1-\eps/4)|A|^3/K$, then there is a 
	set $A'\subseteq A$ of size $|A'|\ge (1-\eps)K^{-1/2}|A|$ such that 
	$|A'-A'|\le 2^{33}\eps^{-9}K^{4}|A'|$.
\end{lem} 

\begin{proof}
	Let $Q=\{d_1, \ldots, d_{|Q|}\}$ enumerate $Q$. By the definition 
	of $Q$ there are real numbers $x_1, \ldots, x_{|Q|}\in [0, 1]$ such that  
		\[
		r_{A-A}(d_i)
		=
		x_iK^{-1/2}|A|
		\quad \text{ holds for every } i\in [|Q|]\,.
	\]
		Owing to $\sum_{d\in A-A} r_{A-A}(d)=|A|^2$ and the hypothesis we have 		\[
		\sum_{i=1}^{|Q|}x_i \le K^{1/2} |A|
		\quad \text{ as well as } \quad
		\sum_{i=1}^{|Q|}x_i^2 \ge (1-\eps/4)|A|\,.
	\]
		By Lemma~\ref{l:ST} applied with $\alpha=1-\eps/4$ there exist an index 
	set $I\subseteq [|Q|]$ such that  
		\begin{equation}\label{eq:3254}
		\sum_{i\in I}x_i
		\ge
		\max\left\{(1-\eps/2)|A|, 
			\bigl(2^{-21}\eps^5K^{-1}|I|^4|A|^2\bigr)^{1/6}\right\}\,.
	\end{equation}
		Now we set $Q'=\{d_i\colon i\in I\}$, consider the relation 
		\[
		R=\{(a_1, a_2)\in A^2\colon a_1-a_2\in Q'\}
	\]
		and define $\delta\in (0, 1]$ by $|R|=\delta |A|^2$.
	Due to
		\[
		\delta
		=
		|A|^{-2}\sum_{i\in I}r_{A-A}(d_i)
		=
		\frac 1{K^{1/2}|A|}\sum_{i\in I} x_i
	\]
		the bounds in~\eqref{eq:3254} imply both
		\begin{equation}\label{eq:3858}
		\delta
		\ge
		(1-\eps/2)K^{-1/2}
		\quad \text{ and } \quad
		\frac{|I|^4}{\delta^6|A|^4}
		\le 
		2^{21}\eps^{-5}K^4\,.
	\end{equation}
		
	By Lemma~\ref{lem:tv} applied to $\xi=\eps/2$ and $R$ there exists a 
	set $A'\subseteq A$ of size 
		\[
		|A'|\ge (1-\eps/2)\delta|A|\ge (1-\eps)K^{-1/2}|A|
	\]
		such that for every pair $(a_1, a_2)\in A'^2$ there are at least 
	$2^{-11}\eps^4\delta^4|A|^2|A'|$ triples $(x, b, y)\in A^3$ with  
	$(a_1, x), (b, x), (b, y), (a_2, y)\in R$.
	Due to the equation 
		\[
		(a_1-a_2)=(a_1-x)-(b-x)+(b-y)-(a_2-y)
	\]
		this means that every difference $a_1-a_2\in A'-A'$ has at least 
	$2^{-11}\eps^4\delta^4|A|^2|A'|$ representations of the form 
	$q_1-q_2+q_3-q_4$ with $q_1, q_2, q_3, q_4\in Q'$, whence
		\[
		|A'-A'|
		\le 
		\frac{|Q'|^4}{2^{-11}\eps^4\delta^4|A|^2|A'|}
		\overset{\eqref{eq:3858}}{\le} 
		2^{32}\eps^{-9}K^4(\delta|A|/|A'|)^2|A'|\,.	
	\]
		Due to $|A'|\ge (1-\eps/2)\delta |A|\ge \delta|A|/\sqrt{2}$ the result follows.  
\end{proof}

\subsection*{Acknowledgement}
We would like to thank the referees for reading our article very carefully.

\end{document}